\documentclass[12pt]{amsart}
\usepackage{amssymb,latexsym,amsfonts}
\usepackage{amsmath}
\usepackage{amsfonts}
\usepackage{amscd}
\usepackage{amsthm}

\newcommand{\DR}{{}^{\ast}\mkern-1.3mu {\mathbf{ R}}\mkern-1.3mu^{\ast}}
\newcommand{\dR}{{}^{\ast}{\mathbf{ R}}^{\ast}}

\newcommand{\cc}{{\mathbf {c}} }

\newcommand{\Scal}{{\mathbf {Scal}}}
\newcommand{\R}{{\rm R}}
\newcommand{\g}{{\mathbf g}}

\newcommand{\Ric}{{\mathbf {Ric}}}

\DeclareMathAlphabet{\mathpzc}{OT1}{pzc}{m}{it}

\def\R{\mathbf R}

\newtheorem{theorem}{Theorem}[section]
\newtheorem{corollary}[theorem]{Corollary}

\newtheorem{lemma}[theorem]{Lemma}

\newtheorem{proposition}[theorem]{Proposition}

\newtheorem{thm}{Theorem}[section]

\newtheorem{definition}[thm]{Definition}

\newtheorem{example}{Example}[section]
\newtheorem{remark}[thm]{Remark}

\begin{document}
\author{Mohammed Larbi Labbi} 
\address{Department of Mathematics\\
 College of Science\\
  University of Bahrain\\
  32038, Bahrain.}
\email{mlabbi@uob.edu.bh}
\renewcommand{\subjclassname}{
  \textup{2020} Mathematics Subject Classification}

\title{Generalized Double Duals of the Riemann Tensor in Geometry and Gravity}  \maketitle

\begin{abstract}
The Riemann curvature tensor fully encodes local geometry, but its Ricci contraction retains only limited information: only the Ricci tensor and the scalar curvature survive, while the Weyl curvature vanishes identically. We show that contracting instead the double dual of the Riemann tensor unlocks the full curvature structure, producing a canonical hierarchy of symmetric, divergence--free $(p,p)$ double forms. These tensors satisfy the first Bianchi identity and obey a hereditary contraction relation interpolating between the double dual tensor and the Einstein tensor.

We prove that, in a generic geometric setting, each tensor in this hierarchy is the unique divergence--free $(p,p)$ double form depending linearly on the Riemann curvature tensor, thereby providing canonical higher--rank parents of the Einstein tensor.

Their sectional curvatures coincide with the $p$--curvatures; notably, the $2$--curvature determines the full Riemann curvature tensor and forces the $\hat A$--genus of a compact spin manifold to vanish when nonnegative, a property not shared by Ricci or scalar curvature. Finally, we extend the construction to Gauss--Kronecker curvature tensors and Lovelock theory, showing in particular that the second Lovelock tensor $T_4$ admits a genuine four--index parent tensor.
\end{abstract}

\tableofcontents

\section{Introduction and statement of the main results}
The purpose of this section is to state the main results of the paper and to motivate the study of higher--rank curvature tensors arising from the Riemann curvature tensor, as well as to explain the role of the double dual construction in organizing curvature information. We begin by recalling the limitations of Ricci contraction as a mechanism for extracting geometric invariants from the Riemann tensor. We then outline the canonical hierarchy of divergence--free curvature tensors introduced in this paper and indicate their geometric, topological, and gravitational significance.

\subsection{Double dual of the Riemann curvature tensor and parents of the Einstein tensor}

The Riemann curvature tensor $\R$ is a $(2,2)$ double form that already encodes
the full local geometry of a Riemannian manifold. The difficulty addressed in
this paper is therefore not a lack of information in $\R$, but rather the
degeneracy of the Ricci contraction as a mechanism for extracting geometric
invariants. Indeed, $\R$ admits only two nontrivial Ricci contractions: the
Ricci tensor $\Ric = \cc\R$ and the scalar curvature $\Scal = \cc^{2}\R$. In
particular, the Weyl tensor, which represents the trace-free part of $\R$,
vanishes identically under all Ricci contractions. As a consequence, any
curvature theory formulated solely in terms of Ricci contractions is blind to
a substantial portion of the curvature information already present in $\R$.

This limitation is not intrinsic to the Riemann tensor itself, but to the Ricci
contraction. It therefore motivates the search for canonical curvature tensors
obtained from $\R$ by operations other than contraction, while preserving the
fundamental algebraic and differential properties required in geometry and
physics. A natural construction in this direction is provided by the double
dual operation 
\[
\dR(\,\cdot\,,\,\cdot\,) := \R(\ast\,\cdot\,,\ast\,\cdot\,),
\]
where the Hodge star acts independently on each argument of the double form
$\R$. Unlike Ricci contraction, this operation does not reduce curvature order
but reorganizes and expands the information contained in $\R$ into a higher-rank
$(n-2,n-2)$ double form. In dimension four, the double dual preserves degree and
plays a fundamental role in geometry and general relativity; in higher
dimensions, it provides a canonical mechanism for accessing curvature
information that is invisible to Ricci contractions alone.

In this paper, we study the successive Ricci contractions of the double dual
tensor $\dR$. For each integer $p$ with $0 \le p \le n-2$, we denote by
\[
\DR_p := \frac{1}{(n-p-2)!}\,\cc^{\,n-p-2}(\dR)
\]
the corresponding symmetric $(p,p)$ double form.
For $p=0$, $\DR_0$ coincides with one half of the scalar curvature.
For $p=1$, $\DR_1$ coincides with the Einstein tensor
$\tfrac12 \Scal \g-\Ric$.
For $p=2$, $\DR_2$ is a symmetric $(2,2)$ double form enjoying the same
algebraic symmetries as the Riemann tensor, including the first Bianchi
identity, and it is free of divergence, making it a natural candidate for
four-index formulations of gravity (see section \ref{motiv}). Moreover, $\DR_2$ has the same Weyl part
as $\R$ (see Corollary \ref{same-weyl}) and determines $\R$ in all dimensions, (see Proposition \ref{dR-to-R}).

In general, all tensors $\DR_p$ are symmetric, satisfy the first Bianchi
identity and are free of divergence, see Corollary \ref{flat}. They satisfy the hereditary contraction
relation
\[
\cc(\DR_p) = (n-p-1)\,\DR_{p-1},
\]
so that the tensors $\DR_p$, for $p \ge 2$, form a hierarchy of
\emph{parent tensors} of the Einstein tensor:

\[\displaystyle{
\dR=\dR_{n-2}\xrightarrow{\cc} \dR_{n-3} \xrightarrow{\cc} ...\xrightarrow{\cc}\dR_2 \xrightarrow{\cc }\dR_{1}={\mathbf{Einstein\,\, tensor}}.}
\]

A central result of this paper (Theorems  \ref{uniqueness-p1} and \ref{uniqueness-thm} )
shows that the tensors $\DR_p$, for $1 \le p \le n-2$, are generically the
\emph{unique} divergence-free $(p,p)$ double forms that depend linearly on the
Riemann curvature tensor.

\subsection{Sectional curvature of the generalized double duals}

We also study the sectional curvatures associated with the tensors $\DR_p$.
These are precisely one half of the $p$-curvatures $s_p$, previously introduced
and studied by the author, see for instance \cite{Labbi-these, Labbi-Toulouse, Labbi-AGAG, Labbi-Botvinnik}. We prove in
Theorem~\ref{fbp} that positivity of $s_p$ satisfies the fiber bundle
property.

Among these invariants, the $2$-curvature $s_2$ plays a particularly important
role. We show that $s_2$ determines the tensor $\DR_2$, and consequently the
full Riemann curvature tensor $\R$
(Proposition~\ref{dR-to-R}). We prove in
Theorem \ref{Ahat} that if $(M^{4k},\g)$ is a compact spin manifold with
nonnegative $2$-curvature, then its $\hat A$-genus vanishes. A property  that does not hold under $\Scal\geq 0$ or even under $\Ric\geq 0$, see example \ref{example1}. 

\subsection{Applications to modified field equations of general relativity}

A four-index formulation of the gravitational field equations in arbitrary dimension was recently proposed by Moulin \cite{Moulin1,Moulin2}, in which the geometric side of the equations is expressed directly in terms of the double dual curvature tensor $\dR_2$. In this approach, contraction of the four-index equations recovers the classical Einstein equation, while retaining additional curvature information prior to contraction.

Such formulations provide a natural setting in which higher rank divergence free curvature tensors arise, and they motivate the study of canonical four index objects satisfying the same structural properties as the Einstein tensor. The conceptual background and geometric motivation for these four index formulations are discussed in detail in Section~\ref{motiv}.

\subsection{Parents of Lovelock tensors and the case of $T_4$}

In Section~\ref{lovelock}, for each integer $q$ such that $2 \le 2q \le n$, we
apply the same analysis to the Gauss--Kronecker curvature tensors $\R^q$, which
are $(2q,2q)$ double forms. We consider their double dual tensors defined by
\[
{}^{\ast}\!\left(\R^q\right)^{\ast}(\,\cdot\,,\,\cdot\,)
:= \R^q(\ast\,\cdot\,,\ast\,\cdot\,),
\]
which are $(n-2q,n-2q)$ double forms. For simplicity, we denote these tensors by
$\ast(\R^q)$ or simply $\ast\,\R^q$ in what follows.

By performing successive Ricci contractions, we define a family of curvature
tensors, called the $(p,q)$-curvature tensors, for $0 \le p \le n-2q$, by
\[
\R^{(p,q)}
:=
\frac{1}{(n-2q-p)!}\,\cc^{\,n-2q-p}\!\left(\ast\,\R^q\right).
\]
For $p=0$, the tensors $\R^{(0,q)}$ are scalar functions and coincide, up to a
constant factor, with the Gauss--Bonnet curvatures of $(M,\g)$. In particular,
when the dimension $n$ is even, $\R^{(0,n/2)}$ is the Gauss--Bonnet integrand.
For $q=1$, we recover the generalized double duals of the Riemann tensor,
namely $\R^{(p,1)} = \dR_p$. For $p=1$, the tensor $\R^{(1,q)}$ coincides with
the Lovelock tensor $T_{2q}$ of order $2q$.

The curvature tensors $\R^{(p,q)}$ are symmetric $(p,p)$ double forms that
satisfy the first Bianchi identity. Although they do not satisfy the second
Bianchi identity in general, they are always free of divergence. Moreover, they
satisfy the following hereditary property with respect to Ricci contraction,
for $p \ge 1$:
\[
\R^{(p,q)}
\xrightarrow{\ \cc\ }
\R^{(p-1,q)}
\xrightarrow{\ \cc\ }
\cdots
\xrightarrow{\ \cc\ }
\R^{(1,q)} = \mathbf{Lovelock\, tensor}.
\]
In particular, $(p-1)$ contractions of $\R^{(p,q)}$ yield the Lovelock tensor
$T_{2q}$.
For this reason, we shall call the tensors $\R^{(p,q)}$ with $p \ge 2$
\emph{$p$-parents of the Lovelock tensor}.

The first Lovelock tensor $T_2$, corresponding to $q=1$, is the usual Einstein
tensor. The second Lovelock tensor $T_4$, obtained for $q=2$, can classically be
recovered from the $(4,4)$ double form $\R^2$ via the formula (see
Equation~\eqref{dRcR} below)
\[
T_4
=
\frac{\cc^{4}\R^2}{4!}\,\g
-
\frac{\cc^{3}\R^2}{3!}.
\]
This requires contracting either the $8$-index tensor $\R^2$ or its double dual
$\ast\,\R^2$, which is a $(2n-8)$-index tensor. We prove in this paper that the
second Lovelock tensor $T_4$ can alternatively be obtained by contracting a
\emph{four-index} curvature tensor, in complete analogy with the way the
Einstein tensor is obtained from the Riemann tensor. More precisely, we show in
Proposition~\ref{lov-index4} that $T_4$ can be obtained by contracting the
$(2,2)$ double form $\R \circ \dR_2$, where $\dR_2$ denotes the generalized
double dual of the Riemann tensor and $\circ$ is the composition product of
double forms:
\[
T_4
=
\frac{1}{2}\,\cc^2(\R \circ \dR_2)\,\g
-
2\,\cc(\R \circ \dR_2).
\]

In the final section, we speculate on possible modifications of Lovelock field
equations for gravity. We first observe that significant components of the
Gauss--Kronecker curvature tensors $\R^q$, including their trace-free parts, do
not contribute to the classical Lovelock field equations. This loss of
information stems from the fact that one must perform $(2q-1)$ contractions of
$\R^q$ in order to recover the Lovelock tensor $T_{2q}$.

We show in Section~\ref{subs} that, in order for all components of $\R^q$ to
contribute effectively to the field equations, one may replace the Lovelock
tensor $T_{2q}$ by its parent tensor $\R^{(2q,q)}$ when $2q \le n/2$, and by
$\R^{(n-2q,q)}$ when $2q > n/2$.

\subsection{Motivation: four-index gravity and the double dual curvature tensor}\label{motiv}

Einstein’s field equations are written in terms of the Einstein tensor
$\frac{1}{2} \Scal\,  \g-\Ric$, which is obtained from the Riemann curvature tensor by
Ricci contraction. As a result, the traceless part of the Riemann tensor, namely
the Weyl tensor, disappears completely in Einstein’s equation, since it
vanishes under all Ricci contractions. This is despite the fact that the Weyl
tensor encodes essential geometric and physical information, including tidal
forces and the free gravitational field.

This observation has motivated the development of gravitational field
equations written directly at the level of four-index curvature tensors. A
natural candidate for such formulations is the double dual $\dR$ of the Riemann
tensor, which has the same algebraic symmetries as $\R$ and is free of
divergence.

In dimension $n=4$, the double dual of the Riemann tensor is classically given
by (see for instance \cite{MTW})
\[
\dR_{ijkl}
=
\frac{1}{4}\,\varepsilon_{ij}{}^{pq}\,\R_{pqrs}\,\varepsilon^{rs}{}_{kl},
\]
and it satisfies the algebraic identity
\[
\dR_{ijkl}
=
\R_{ijkl}
-
(\g_{ik}\Ric_{jl}-\g_{il}\Ric_{jk}+\g_{jl}\Ric_{ik}-\g_{jk}\Ric_{il})
+
\frac12(\g_{ik}\g_{jl}-\g_{il}\g_{jk})\Scal.
\]
The tensor $\dR_{ijkl}$ has the same algebraic symmetries as the Riemann tensor,
satisfies the first Bianchi identity, and is free of divergence. Its contraction
reproduces the Einstein tensor, making it a natural four-index analogue of the
Einstein tensor.

Recently, Moulin  \cite{Moulin1, Moulin2}, proposed a four-index reformulation of Einstein’s field
equations in arbitrary dimension, in which the geometric side of the field
equation is given by the double dual tensor $\dR_2$. The contracted equation is
equivalent to the classical Einstein equation, while the traceless part yields
an additional constraint involving the Weyl tensor.

Earlier, Zakir \cite{Zakir} emphasized that an essential part of the gravitational field is
lost when passing to two-index field equations and advocated writing
gravitational field equations directly at the four-index level. The results of
the present paper provide a natural geometric framework for these approaches by
embedding the double dual tensor into a canonical hierarchy of divergence-free
curvature tensors and extending the construction to Lovelock theory.
\bigskip

Throughout the paper, all manifolds, metrics, and curvature tensors are assumed to be smooth, the Riemannian metric is taken to be positive definite, and $n$ denotes the dimension of the underlying manifold. All statements are therefore formulated in the Riemannian setting; nevertheless, the results of Sections~2,~3, and~5 are largely algebraic and could be extended to the pseudo-Riemannian case with minor and standard modifications. By contrast, the notion of sectional curvature and the associated positivity results developed in Section~4 are inherently Riemannian.

\section{The double dual curvature tensor and its contractions}

This section develops the algebraic and differential properties of the double dual of the Riemann curvature tensor viewed as a double form. Our aim is to construct a canonical hierarchy of symmetric $(p,p)$ double forms obtained by successive contractions, and to characterize these tensors intrinsically by their divergence--free property. The main result of the section shows that, under mild and generic geometric assumptions, this hierarchy exhausts all divergence--free $(p,p)$ double forms depending linearly on the Riemann curvature tensor.\\

Let $(M^n,\g)$ be a Riemannian manifold with $n \ge 4$. At each point $p \in M$,
the $(0,4)$ Riemann curvature tensor $R_{ijkl}$ may be viewed as an element of
$\Lambda^2 T^*_p M \otimes \Lambda^2 T^*_p M$, or equivalently as a symmetric
$(2,2)$ double form on $T_p M$. The \emph{double dual of the Riemann tensor} is
well known in the physics literature in dimension four. It is obtained by
applying the Hodge star operator to each antisymmetric index pair of the Riemann
curvature tensor and often denoted by left and right stars especially in physics literature:
\[
\dR(\,\cdot\,,\,\cdot\,) := \R(\ast\,\cdot\,,\ast\,\cdot\,),
\]
where the Hodge star acts independently on each argument of the double form
$\R$. In dimension four, its components are given by (see
\cite[page 325]{MTW}) 
\begin{equation}
(\DR)_{ijkl}
=
\frac{1}{4}\,\sum_{p,q,r,s}\varepsilon_{ijpq}\,\varepsilon_{klrs}\,R_{pqrs},
\end{equation}
where $\varepsilon_{ijpq}$ denotes the Levi--Civita symbol. The tensor $\DR$
shares the same algebraic symmetries as the Riemann tensor $\R$ and satisfies
the first Bianchi identity. Although it does not satisfy the second Bianchi
identity in general, it is always free of divergence
\cite[page 325]{MTW}. Moreover, its first contraction yields the Einstein
tensor, while its second contraction yields the scalar curvature.

The above expression for the components of $\DR$ does not generalize directly
to higher dimensions, since the symbols $\varepsilon_{ijpq}$ are only defined
when $n=4$. However, one may rewrite the formula using the generalized
Kronecker delta as
\begin{equation}
(\DR)_{ijkl}
=
\frac{1}{4}\,\sum_{p,q,r,s}\delta^{ijpq}_{klrs}\, R_{pqrs}.
\end{equation}
This expression is valid in all dimensions $n \ge 4$ and therefore provides a
natural definition of the double dual of the Riemann tensor in arbitrary
dimension and shall henceforth be denoted  as $\dR_2$. The resulting tensor enjoys the same algebraic and differential
properties as in dimension four. These properties motivate the introduction of
higher-index analogues whose successive contractions reproduce classical
curvature tensors while preserving divergence-free structure.

\[
\begin{array}{ccc}
\textbf{Riemann tensor } \R & & \textbf{Double dual Riemann tensor } \DR_2 \\[0.2cm]
\text{Symmetric $(2,2)$ double form,} & & \text{Symmetric $(2,2)$ double form,} \\
\text{satisfies Bianchi I and II identities} & &
\text{satisfies Bianchi I identity,} \\
& & \text{free of divergence} \\[0.2cm]
\Big\downarrow{\scriptstyle\cc} & & \Big\downarrow{\scriptstyle\cc} \\
\Ric & & \mathrm{Einstein} = \tfrac12\Scal\,\g - \Ric \\[0.2cm]
\Big\downarrow{\scriptstyle\cc} & & \Big\downarrow{\scriptstyle\cc} \\
\Scal & & \tfrac{n-2}{2}\Scal
\end{array}
\]
Here $\cc$ denotes the contraction map. Beyond their intrinsic geometric
interest, the double dual tensor $\dR_2$ also plays a role in gravitational
theory. A modern perspective was recently provided by Moulin
\cite{Moulin1,Moulin2}, who developed a four-index formulation of the
gravitational field equations in all dimensions $n \ge 4$ whose contraction
recovers the classical Einstein equation. In this formulation, and in contrast
to the classical Einstein equation, the full curvature tensor plays a central
role through its double dual $\dR_2$.

\subsection{Higher index double dual tensors of Riemann}

\begin{definition}
For each integer $p$ with $0 \le p \le n-2$, the \emph{generalized double dual
of the Riemann tensor} of order $p$, denoted by
$\dR_p = (\dR_{i_1\ldots i_p,j_1\ldots j_p})$, is the covariant tensor of order
$2p$ whose components are given by
\begin{equation}
\dR_{i_1\ldots i_p,j_1\ldots j_p}
=
\frac{1}{4}\,\sum_{p,q,r,s}
\delta^{pqi_1\ldots i_p}_{rsj_1\ldots j_p}\, R_{pqrs}.
\end{equation}
For $p=0$, we set
\[
\dR_0
=
\frac{1}{4}\,\sum_{p,q,r,s}\delta^{pq}_{rs}\, R_{pqrs}
=
\frac{\Scal}{2}.
\]
\end{definition}

These tensors are obtained by reorganizing and expanding the curvature information of the
Riemann tensor through double duality and contraction. For $p=1$, the tensor
$\dR_1$ coincides with the classical Einstein tensor. For $p=2$,  we recover the double dual tensor $\dR_2$ of the previous section.

\subsection{Index-free formulation of the higher double duals of Riemann}

The following proposition provides an index-free formula for the higher double
duals of the Riemann tensor using the exterior product of double forms.\\
We emphasize that its proof relies on a systematic passage from index-based expressions to index-free identities using the formalism of double forms. This approach allows one to reinterpret classical curvature contractions in a conceptual and coordinate-free manner.

\begin{proposition}\label{index-free}
For each $p$ with $0 \le p \le n-2$, one has
\[
\dR_p
=
\ast\,\frac{\g^{\,n-p-2}\,\R}{(n-p-2)!}.
\]
\end{proposition}

\begin{proof}
Let $\{e_1,\ldots,e_n\}$ be a local orthonormal frame. Using the standard
identification between the generalized Kronecker delta and exterior powers of
the metric tensor, we compute
\begin{equation}
\begin{split}
\dR_{i_1\ldots i_p,j_1\ldots j_p}
=&
\frac{1}{4}\,\sum_{p,q,r,s}
\delta^{pqi_1\ldots i_p}_{rsj_1\ldots j_p}\, \R_{pqrs} \\
=&
\frac{1}{4}\,\sum_{p,q,r,s}
\frac{\g^{p+2}}{(p+2)!}
(e_p,e_q,e_{i_1},\ldots,e_{i_p};e_r,e_s,e_{j_1},\ldots,e_{j_p})\, \R_{pqrs} \\
=&
\frac{1}{4}\,
\left(
i_{\R}\frac{\g^{p+2}}{(p+2)!}
\right)
(e_{i_1},\ldots,e_{i_p};e_{j_1},\ldots,e_{j_p}).
\end{split}
\end{equation}
Consequently,
\[
\dR_p
=
\frac{1}{4}\, i_{\R}\frac{\g^{p+2}}{(p+2)!}
=
\frac{1}{4}\,\ast\,\R\ast\!\left(\frac{\g^{p+2}}{(p+2)!}\right)
=
\ast\,\frac{\g^{n-p-2}\,\R}{(n-p-2)!}.
\]
Here $i_\omega$ denotes the interior product of double forms, which satisfies
$i_{\R} = \ast\,\R\ast$; see
\cite{Labbi-Bel,Labbi-Lapl-Lich}. This completes the proof.
\end{proof}
\begin{corollary}\label{flat}
\begin{enumerate}
\item
For each $p$ with $1 \le p \le n-2$, the tensors $\dR_p$ are free of divergence.
\item
For each $p$ with $2 \le p \le n-2$, one has
\[
\dR_p = 0 \iff \R = 0.
\]
\end{enumerate}
\end{corollary}

\begin{proof}
The first assertion follows from Proposition~\ref{index-free} and the identity
$\ast\,\delta\,\ast(\omega) = (-1)^r D\omega$ for an $(r,r)$ double form (see,
for instance, identity~(49) in \cite{Labbi-Lapl-Lich}), where $D$ denotes the
second Bianchi sum, $\delta$ the divergence operator, and
$\omega = \g^{n-p-2}\R$, which satisfies $D\omega = 0$.

The second assertion follows directly from Proposition~\ref{index-free} and the
cancellation property of exterior multiplication by powers of the metric
tensor $\g$; see \cite[Proposition~2.3]{Labbi-double-forms}.
\end{proof}
\begin{corollary}
For each $p$ with $1 \le p \le n-2$, one has
\begin{equation}\label{contract}
\dR_p
=
\frac{1}{(n-p-2)!}\,\cc^{\,n-p-2}(\ast\,\R).
\end{equation}
In particular,
\begin{equation}\label{trace}
\cc(\dR_p) = (n-p-1)\dR_{p-1}.
\end{equation}
\end{corollary}

\begin{proof}
This follows directly from Proposition~\ref{index-free} and the identity
$\cc^{\,r} = \ast\,\g^{\,r}\ast$.
\end{proof}
We have therefore obtained a sequence of higher-rank tensors that all share the
algebraic symmetries of the Riemann tensor, satisfy the first Bianchi identity,
and are free of divergence. Moreover, successive contractions recover the
Einstein tensor and the scalar curvature:
\[
\dR
=
\dR_{n-2}
\xrightarrow{\cc}
\dR_{n-3}
\xrightarrow{\cc}
\cdots
\xrightarrow{\cc}
\dR_1
=
\text{Einstein tensor}
\xrightarrow{\cc}
\Scal.
\]

The following property shows that, at each point of the manifold, the angle
(with respect to the natural inner product on double forms) between the
Riemann tensor $\R$ and its double dual $\dR_2$ is determined by the
Gauss--Bonnet curvature
$h_4 = \|\R\|^2 - \|\Ric\|^2 + \tfrac14 \Scal^2$.

\begin{proposition}
One has
\[
\langle \R, \dR_2 \rangle = h_4.
\]
\end{proposition}

\begin{proof}
This follows directly from the properties of the inner product of double forms as follows
\[
\langle \R, \dR_2 \rangle
=
\ast\!\left( (\ast\,\dR_2)\R \right)
=
\ast\!\left( \frac{\g^{n-4}\R^2}{(n-4)!} \right)
=
h_4.
\]
\end{proof}
\subsection{A uniqueness theorem for the generalized double duals of the Riemann tensor}

Let $\R$ denote the $(0,4)$ Riemann curvature tensor of a Riemannian manifold
$(M,\g)$, viewed as a symmetric $(2,2)$ double form, and let $\omega$ be a
$(p,p)$ double form on $M$.

\begin{definition}
A $(p,p)$ double form $\omega$ is said to be a \emph{linear polynomial with
respect to the Riemann tensor $\R$} if its components
$\omega_{i_1\ldots i_p,j_1\ldots j_p}$ with respect to any orthonormal frame are
given by universal (that is, frame-independent) homogeneous linear polynomials
in the components $R_{ijkl}$ of the Riemann curvature tensor.
\end{definition}

Typical examples of $(p,p)$ double forms that are linear polynomials in $\R$
include $\g^p\,\Scal$, $\g^{p-1}\Ric$, and $\g^{p-2}\R$. Using invariant theory,
Bivens~\cite{Bivens} showed that these three tensors span the entire space of
$\R$-linear polynomial $(p,p)$ double forms. More precisely, the following
result holds.

\begin{theorem}[{\cite[Theorem~3.1]{Bivens}}]
Let $(M,\g)$ be a Riemannian manifold of dimension $n \ge 4$, let $\R$ be its
Riemann curvature tensor, and let $p \ge 1$. The space of all $\R$-linear
polynomial $(p,p)$ double forms is a real vector space spanned by
\[
\{\g^p\,\Scal,\ \g^{p-1}\Ric,\ \g^{p-2}\R\}.
\]
Here we adopt the convention that $\g^{-1}=0$ and $\g^{0}=1$.
\end{theorem}

We use this result to prove that, in a generic geometric setting, the
generalized double dual tensors $\dR_p$ are the \emph{unique} divergence-free
linear polynomials in the Riemann tensor. We first treat the case $p=1$ and
then address higher values of $p$.

\begin{theorem}\label{uniqueness-p1}
Let $(M,\g)$ be a Riemannian manifold of dimension $n \ge 4$, let $\R$ be its
Riemann curvature tensor, and let
$\dR_1 = \tfrac12\,\Scal\,\g - \Ric$ denote the Einstein tensor. The space of all
divergence-free $\R$-linear polynomial $(1,1)$ double forms on $(M,\g)$ is a
real vector space spanned by
\begin{itemize}
\item[a)]
$\{\dR_1\}$ if the scalar curvature of $(M,\g)$ is not constant;
\item[b)]
$\{\dR_1,\ \Scal\,\g\}$ if the scalar curvature of $(M,\g)$ is constant.
\end{itemize}
\end{theorem}

\begin{proof}
Let $\omega$ be a divergence-free $\R$-linear polynomial $(1,1)$ double form.
By Bivens’ theorem, $\omega$ can be written as
\[
\omega
=
\alpha\,\g\,\Scal + \beta\,\Ric
=
(\alpha+\tfrac{\beta}{2})\,\g\,\Scal - \beta\,\dR_1,
\]
for some real constants $\alpha$ and $\beta$. Since $\dR_1$ is divergence
free, we obtain
\[
\delta\omega = 0
\iff
(\alpha+\tfrac{\beta}{2})\,\delta(\g\,\Scal)
=
-(\alpha+\tfrac{\beta}{2})\,\widetilde D\Scal
=
0.
\]
This proves the result.\end{proof}

\begin{theorem}\label{uniqueness-thm}
Let $(M,\g)$ be a Riemannian manifold of dimension $n \ge 4$, let $\R$ be its
Riemann curvature tensor, and let $2 \le p \le n-2$. The space of all
divergence-free $\R$-linear polynomial $(p,p)$ double forms on $(M,\g)$ is a
real vector space spanned by
\begin{itemize}
\item[a)]
$\{\dR_p\}$ in the generic case where the Weyl tensor is not harmonic and the
scalar curvature is not constant;
\item[b)]
$\{\dR_p,\ \ast(\g^{n-p-1}\Ric),\ \ast(\g^{n-p}\Scal)\}$ when the Weyl tensor is
harmonic and the scalar curvature is constant;
\item[c)]
$\{\dR_p,\ \ast(\g^{n-p-1}A)\}$ when the Weyl tensor is harmonic and the scalar
curvature is not constant, where $A$ denotes the Schouten tensor;
\item[d)]
$\{\dR_p,\ \ast(\g^{n-p}\Scal)\}$ when the Weyl tensor is not harmonic and the
scalar curvature is constant.
\end{itemize}
\end{theorem}

\begin{proof}
Recall that a double form $\theta$ is divergence free if and only if its double
dual $\star\theta$ is Codazzi, that is, satisfies $D(\star\theta)=0$. It
therefore suffices to classify all Codazzi $\R$-linear polynomial $(p,p)$
double forms.

Let $\omega$ be such a form. By Bivens’ theorem, we may write
\[
\omega
=
\alpha\,\g^p\,\Scal
+
\beta\,\g^{p-1}\Ric
+
\gamma\,\g^{p-2}\R,
\]
for real constants $\alpha$, $\beta$, and $\gamma$. The condition $D\omega=0$
is equivalent to
\[
D\!\left[
\g^{p-1}(\alpha\,\g\,\Scal+\beta\,\Ric)
\right]
=
(-1)^{p-1}\,\g^{p-1}\,D(\alpha\,\g\,\Scal+\beta\,\Ric)
=
0.
\]
Since $p \le n-2$, we have $p-1+2+1<n+1$, and therefore by
\cite[Proposition~2.3]{Labbi-double-forms} we obtain
\begin{equation}\label{star}
D\omega = 0
\iff
D(\alpha\,\g\,\Scal+\beta\,\Ric) = 0.
\end{equation}

The Schouten tensor $A$ is related to the Ricci tensor by
\[
\Ric = (n-2)A + \frac{\Scal}{2(n-1)}\,\g.
\]
Substituting this into~\eqref{star} yields the equivalent condition
\begin{equation}\label{stars}
D\omega = 0
\iff
D\!\left[
\Bigl(\alpha+\frac{\beta}{2(n-1)}\Bigr)\,\g\,\Scal
+
\beta\,(n-2)A
\right]
=
0.
\end{equation}
Since the Cotton tensor $DA$ is trace free, taking the trace of the above
equation gives
\[
\Bigl(\alpha+\frac{\beta}{2(n-1)}\Bigr)\,n\,D(\Scal) = 0.
\]

From these relations, the following cases arise:
\begin{itemize}
\item
In the generic case where $DA\not\equiv 0$ and the scalar curvature is not
constant, one has $\alpha=\beta=0$, and hence
$\omega=\g^{p-2}\R$.
\item
If $DA=0$ and the scalar curvature is constant, then $D\omega=0$ for all real
values of $\alpha$, $\beta$, and $\gamma$.
\item
If $DA=0$ and the scalar curvature is not constant, then
$\beta=-2(n-1)\alpha$, and hence
\[
\omega
=
\alpha\,\g^{p-1}(\g\,\Scal-2(n-1)\Ric)
+
\gamma\,\g^{p-2}\R
=
-2(n-1)(n-2)\alpha\,\g^{p-1}A
+
\gamma\,\g^{p-2}\R.
\]
\item
If $DA\not\equiv 0$ and the scalar curvature is constant, then $\beta=0$, and
hence
\[
\omega
=
\alpha\,\g^p\,\cc^2\R
+
\gamma\,\g^{p-2}\R.
\]
\end{itemize}
Finally, recall that the condition $DA=0$ is equivalent to the Weyl tensor being
divergence free, and hence harmonic. This completes the proof.\end{proof}

\begin{remark}[Relation to Lovelock's theorem on divergence-free tensors]\label{lovelock-relation}

Theorems \ref{uniqueness-p1} and \ref{uniqueness-thm} provide a linear, higher-rank analogue of Lovelock's celebrated 
uniqueness theorem for divergence-free tensors \cite{Lovelock}. Lovelock's theorem 
classifies the most general symmetric, divergence-free second-rank tensors constructed 
from the metric and its first two derivatives yielding the Lovelock tensors $T_{2k}$, which include the metric  $T_0 = \mathbf{\g}$ and the Einstein tensor $T_2$).

Our results classify, for each $p$, the unique symmetric, divergence-free $(p,p)$ double form that is \emph{linear} in the Riemann curvature tensor $\R$. 
The metric $\g$ does not appear in our classification precisely because it is not linear in $\R$; it enters Lovelock's theorem as the term $T_0=\g$.

For $p=1$, we recover the Einstein tensor (and the metric only in the case when the scalar 
curvature is constant). For $p\geq 2$, we obtain the new higher-rank tensors $\dR_p$, forming a canonical hierarchy of parents of the Einstein tensor. In Section \ref{lovelock}, we extend this construction to obtain $p$-parents for all Lovelock tensors $T_{2q}$.
\end{remark}

\subsection{Self and anti-self duality of $\dR_p$}

The higher double duals of the Riemann tensor provide a characterization of
Einstein metrics and conformally flat metrics with zero scalar curvature in all
even dimensions, as follows.

\begin{theorem}[\cite{Labbi-double-forms, Labbi-these}]\label{self}
Let $(M,\g)$ be a Riemannian manifold of even dimension $n=2p \ge 4$. Define
$\ast(\dR_p)(\cdot,\cdot) := \dR_p(\ast\cdot,\ast\cdot)$. Then:
\begin{itemize}
\item[a)]
$(M,\g)$ is Einstein if and only if $\dR_p$ is self-dual, that is,
\[
\ast(\dR_p) = \dR_p;
\]
\item[b)]
$(M,\g)$ is conformally flat with zero scalar curvature if and only if $\dR_p$
is anti-self-dual, that is,
\[
\ast(\dR_p) = -\dR_p.
\]
\end{itemize}
\end{theorem}

\subsection{Orthogonal decomposition of the tensors $\dR_p$}

The following useful formula follows directly from Formula~(15) in
\cite{Labbi-double-forms}.

\begin{proposition}\label{p-curvature-formula}
For $0 \le p \le n-2$, one has
\begin{equation}\label{dRpp}
\dR_p
=
\ast\!\left(\frac{\g^{n-p-2}\R}{(n-p-2)!}\right)
=
\frac{\g^{p-2}\R}{(p-2)!}
-
\frac{\g^{p-1}\Ric}{(p-1)!}
+
\frac{1}{2}\frac{\g^{p}}{p!}\Scal.
\end{equation}
Here we use the convention that $\g^{-1}=0$ and $\g^0=1$. In particular, $\dR_1$
coincides with the Einstein tensor and $\dR_2$ satisfies the following identity
\begin{equation}\label{dR2}
\dR_2 = \R - \g\,\Ric + \frac{1}{4}\,\g^2\,\Scal.
\end{equation}
\end{proposition}

Equation~\eqref{dR2} reduces in dimension $4$ to the classical Ruse--Lanczos
identity \cite[Proposition~3.1]{GouOtt} and is equivalent in higher dimensions
to the generalized Ruse--Lanczos identity of \cite[equation~A.14]{Senovilla}.

Recall the orthogonal decomposition of the Riemann tensor into irreducible
components,
\[
\R = \omega_2 + \g\,\omega_1 + \g^2\,\omega_0,
\]
where $\omega_2 = W$ is the Weyl tensor,
$\omega_1 = \frac{1}{n-2}(\Ric - \frac{\Scal}{n}\g)$ is the trace-free Ricci
component, and $\omega_0 = \frac{\Scal}{2n(n-1)}$.

It follows that the generalized double dual tensors decompose into orthogonal
components as follows.

\begin{corollary}\label{same-weyl}
For $0 \le p \le n-2$, one has
\[
\dR_p
=
\frac{1}{(p-2)!}\,\g^{p-2}\omega_2
-
\frac{n-p-1}{(p-1)!}\,\g^{p-1}\omega_1
+
\frac{(n-p)(n-p-1)}{p!}\,\g^{p}\omega_0.
\]
We use the convention that $\g^{-1}=0$ and $\g^0=1$. In particular, the Riemann tensor $\R$ and its double dual $\dR_2$ have the same
Weyl (trace-free) component. More precisely,
\[
\dR_2
=
\omega_2
-
(n-3)\,\g\,\omega_1
+
\frac{(n-2)(n-3)}{2}\,\g^2\,\omega_0.
\]
\end{corollary}
\section{Sectional curvatures of the generalized double duals of Riemann}
In this section, we introduce the  sectional curvature associated with the generalized double duals of the Riemann curvature tensor. In complete analogy with the classical sectional curvature of the Riemann tensor, these invariants are defined on Grassmannian bundles of tangent planes and capture curvature information through the evaluation of higher–order curvature tensors on orthonormal frames. We investigate their basic properties and show that they encode significant geometric information, including positivity and non-negativity properties with strong geometric and topological consequences.

\begin{definition}
The \emph{sectional curvature} of the double dual tensor $\dR_p$, denoted by $s_p$ and referred to as the \emph{$p$-curvature}, is a function defined on the Grassmannian bundle of tangent $p$-planes.  
For a tangent $p$-plane $P$, and for any orthonormal basis $\{e_1,\dots,e_p\}$ of $P$, we define
\[
s_p(P):=2\,\dR_p(e_1,\dots,e_p;e_1,\dots,e_p).
\]
\end{definition}

The factor $2$ is introduced so that this definition coincides with the $p$-curvature  that was extensively studied by the author, beginning with his 1995 PhD thesis \cite{Labbi-these}, and subsequently in a series of papers (see for instance \cite{Labbi-Toulouse, Labbi-AGAG, Labbi-Botvinnik}).  
In what follows, we establish  new properties of these invariants that were not covered in the earlier works.

\begin{proposition}\label{dR-to-R}
\begin{enumerate}
\item The double dual tensor $\dR_2$ determines the Riemann curvature tensor $\R$. More precisely,
\begin{equation}
\R=\dR_2-\frac{1}{n-3}\g \, \cc( \dR_2)+\frac{1}{2(n-2)(n-3)}\g\, \cc^{2}(\dR_2).
\end{equation}
\item The sectional curvature $s_2$ of the double dual $\dR_2$ determines the Riemann curvature tensor $\R$.
\end{enumerate}
\end{proposition}

\begin{proof}
The first assertion follows directly by contracting the identity of Proposition~\ref{p-curvature-formula}. Indeed, one obtains
\[
\cc (\dR_2)=(n-3)\Big(\frac{1}{2}\Scal\, \g-\Ric\Big),\qquad 
\cc^{2}(\dR_2)=\frac{(n-2)(n-3)}{2}\Scal.
\]
The second assertion follows from a standard argument, analogous to the classical case of $\R$, showing that the sectional curvature $s_2$ determines the tensor $\dR_2$.  
Recall that $\dR_2$ satisfies the first Bianchi identity. The conclusion then follows from the first part of the proposition.
\end{proof}

As a consequence, all higher double duals $\dR_p$, for $p>2$, are determined by $\dR_2$, or even by its sectional curvature $s_2$.

\medskip

Taking the sectional curvature of both sides of the identity in Proposition~\ref{p-curvature-formula}, we obtain, for $0\le p\le n-2$,
\begin{equation*}
\begin{split}
\frac{1}{2}s_p
&=\frac{1}{(n-p-2)!}\ast\big(\g^{n-p-2}\R\big)(e_1\wedge\cdots\wedge e_p; e_1\wedge\cdots\wedge e_p)\\
&=\frac{\g^{p-2}\R}{(p-2)!}(e_1\wedge\cdots\wedge e_p; e_1\wedge\cdots\wedge e_p)
-\frac{\g^{p-1}\Ric}{(p-1)!}(e_1\wedge\cdots\wedge e_p; e_1\wedge\cdots\wedge e_p)\\
&\hspace{4.5cm}
+\frac{\Scal}{2}\frac{\g^{p}}{p!}(e_1\wedge\cdots\wedge e_p; e_1\wedge\cdots\wedge e_p)\\
&=\sum_{1\le i<j\le p}R(e_i,e_j,e_i,e_j)
-\sum_{i=1}^p\Ric(e_i,e_i)+\frac{1}{2}\Scal.
\end{split}
\end{equation*}
Here, the vectors $\{e_1,\dots,e_p\}$ form an orthonormal family. We have therefore proved the following useful formula.

\begin{lemma}
For $2\le p\le n-2$, one has
\begin{equation}\label{new-formula}
s_p(e_1,\dots,e_p)
=\Scal-2\sum_{i=1}^p\Ric(e_i,e_i)
+2\sum_{1\le i<j\le p}K(e_i,e_j).
\end{equation}
\end{lemma}

\begin{remark}[$C_p$-curvatures.]
Recall that the $C_p$-curvature invariants introduced by Brendle--Hirsch--Johne \cite{BHJ} are defined at a tangent $p$-plane $P$ by
\[
2C_p(P)=\Scal-s_p(P).
\]
Consequently, one obtains the equivalent expression
\begin{equation}
C_p(P)(e_1,\dots,e_p)
=\sum_{i=1}^p\Ric(e_i,e_i)
-\sum_{1\le i<j\le p}K(e_i,e_j).
\end{equation}
Thus, $C_p(P)$ measures the difference between the intermediate $p$-Ricci curvature of $P$ and the $(n-p)$-curvature of the orthogonal $(n-p)$-plane.

Another observation is that $C_p$ arises as the sectional curvature of the curvature tensor
\[
\frac{\g^{p-1}\Ric}{(p-1)!}-\frac{\g^{p-2}\R}{(p-2)!},
\]
which satisfies the first Bianchi identity. This tensor is closely related to the Weitzenb\"ock curvature tensor \cite{Labbi-nachr}; the two differ only by a factor~$2$:
\[
\frac{\g^{p-1}\Ric}{(p-1)!}-2\frac{\g^{p-2}\R}{(p-2)!}.
\]
\end{remark}
\subsection{Positivity properties of the $p$-curvatures}

As explained in \cite{Labbi-strat}, curvature conditions that are stable under surgeries, satisfy the ideal property with respect to Cartesian products, and enjoy a fiber bundle property are particularly subtle and deserve special attention.  
These three properties are fundamental in converting geometric classification problems into topological ones.  
This perspective partly explains why these curvatures have attracted sustained interest and were studied extensively, for instance in \cite{Labbi-these, Labbi-AGAG, Labbi-Botvinnik, BSW}.

Surgery stability was established in \cite{Labbi-AGAG}. Below, we prove that positivity of the sectional curvature $s_p$ of $\dR_p$ on a compact manifold satisfies a fiber bundle property and, consequently, the ideal property.

\subsection{Fiber bundle property}

\begin{theorem}\label{fbp}
Let $M$ be compact and let $\pi:(M,g)\to (B,\check g)$ be a Riemannian submersion with totally geodesic fibers $F_x$, $x\in B$.  
Assume that $\dim F_x\ge p+2$ for some $p\ge 0$.  
If the induced metric on the fibers has positive $p$-curvature, then the total space $M$ admits a Riemannian metric with positive $p$-curvature.
\end{theorem}

This result generalizes Theorem~2.2 of \cite{Labbi-Toulouse} and Lemma~3.2 of \cite{Labbi-AGAG}, where additional assumptions were imposed on the fibers.

\begin{proof}
We consider the canonical variation $g_t$ of $g$, obtained by rescaling the metric in the vertical directions by the factor $t^2$.  
We show that there exists $t>0$ such that the $p$-curvature of $g_t$ is positive.

All curvature invariants of $g_t$ are indexed by $t$, and invariants of the fibers endowed with the induced metric are denoted with a hat. We omit the index $t$ when $t=1$.

Let $P$ be a tangent $p$-plane in $M$. There exist $g_t$-orthonormal vectors $\{E_1,\dots,E_p\}$ spanning $P$, together with $g_t$-orthonormal vertical vectors $U_i$ and horizontal vectors $X_i$, such that
\[
E_i=\cos(\theta_i)U_i+\sin(\theta_i)X_i,\qquad
\theta_i\in[0,\pi/2].
\]
The angles $\theta_i$ are the principal (or Jordan) angles of $P$.

By O'Neill's formulas for Riemannian submersions (see Chapter~9 of \cite{Besse} or Chapter~2 of \cite{Labbi-these}), the sectional curvature of $g_t$ satisfies
\[
K_t(E_i,E_j)
=\frac{\cos^2(\theta_i)\cos^2(\theta_j)}{t^2}\,\hat K(U_i,U_j)
+\mathcal O\!\left(\frac{1}{t}\right).
\]
Consequently, the scalar and Ricci curvatures satisfy \cite{Labbi-Dedicata}
\[
\Scal_t=\frac{1}{t^2}\hat\Scal+\mathcal O\!\left(\frac{1}{t}\right),\qquad
\Ric_t(E_i,E_i)=\frac{\cos^2(\theta_i)}{t^2}\hat\Ric(tU_i,tU_i)+\mathcal O\!\left(\frac{1}{t}\right).
\]

Using Formula~\eqref{new-formula}, we obtain
\begin{equation*}
\begin{split}
(s_p)_t&(P)
=\Scal_t-2\sum_{i=1}^p\Ric_t(E_i,E_i)+2\sum_{1\le i<j\le p}K_t(E_i,E_j)\\
&=\frac{1}{t^2}\hat\Scal
-2\sum_{i=1}^p\frac{\cos^2(\theta_i)}{t^2}\hat\Ric(tU_i,tU_i)
+2\sum_{1\le i<j\le p}\frac{\cos^2(\theta_i)\cos^2(\theta_j)}{t^2}\hat K(U_i,U_j)
+\mathcal O\!\left(\frac{1}{t}\right).
\end{split}
\end{equation*}

Since $\theta_i\in[0,\pi/2]$ and the expression is linear in each $\cos^2(\theta_i)$, its minimum is attained at the corners of the cube where each $\theta_i$ equals either $0$ or $\pi/2$.  

If all $\theta_i=0$, then for sufficiently small $t$,
\[
(s_p)_t(E_1,\dots,E_p)
=\frac{1}{t^2}\hat s_p(tU_1,\dots,tU_p)+\mathcal O\!\left(\frac{1}{t}\right)>0.
\]
At a general corner with exactly $k$ angles equal to $0$ and the remaining $p-k$ equal to $\pi/2$, one obtains
\[
(s_p)_t=\frac{1}{t^2}\hat s_k+\mathcal O\!\left(\frac{1}{t}\right)>0,
\]
for sufficiently small $t$. Since $s_p>0$ implies $s_k>0$ for $0\le k<p\le n-2$, the proof is complete.
\end{proof}

\begin{remark}
The above theorem remains valid when positivity is replaced by negativity of the $p$-curvature.
\end{remark}

\begin{theorem}[Vanishing of the $\hat A$-genus under $s_2\ge 0$]\label{Ahat}
Let $(M^{n},g)$ be a closed spin manifold of dimension $n=4k$ with nonnegative $2$-curvature $s_2\ge 0$. Then
\[
\hat A(M)=0.
\]
\end{theorem}

\begin{proof}
Equation~\eqref{trace} yields $\cc^2(\dR_2)=(n-3)(n-2)\dR_0$, and hence
\[
\Scal=\frac{1}{(n-2)(n-3)}\sum_{i,j}s_2(e_i,e_j),
\]
where $\{e_i\}$ is an orthonormal basis of $T_pM$. Thus $s_2\ge 0$ implies $\Scal\ge 0$.

If $\Scal\not\equiv 0$, then $\Scal>0$ at some point. By the Kazdan--Warner theorem, the metric can be deformed to one of everywhere positive scalar curvature, and it follows from Lichn'erowicz's theorem that $\hat A(M)=0$.

If $\Scal\equiv 0$, then the above identity forces $s_2\equiv 0$. By Corollary~\ref{flat}, the manifold $(M,g)$ is flat. By Bieberbach's theorem, $M$ admits a finite covering by a flat torus $T^{4k}$. Since $\hat A(T^{4k})=0$ and the $\hat A$-genus is multiplicative under finite coverings, the result follows.
\end{proof}

The following example shows that the conclusion fails under the  assumption $\Ric\ge 0$.

\begin{example}[Nonnegative Ricci curvature does not force $\hat A=0$]\label{example1}
Let $X$ be a K3 surface. Then $X$ is spin and admits a Ricci--flat K\"ahler metric, hence $\Ric\ge 0$, while $\hat A(X)=2\neq 0$.
For any $m\ge 1$, consider the product
\[
M = \underbrace{X\times\cdots\times X}_{m\ \text{factors}},
\]
endowed with the product metric. Since each factor is Ricci--flat, one has $\Ric_M\equiv 0$.  
On the other hand, multiplicativity of the $\hat A$-genus gives
\[
\hat A(M)=\hat A(X)^m=2^m\neq 0.
\]
Thus there exist closed spin manifolds with $\Ric\ge 0$ but $\hat A\ne 0$. In particular, such a manifold cannot admit a Riemannian metric with nonnegative $2$-curvature.
\end{example}
\begin{remark}[Comparison with scalar curvature vanishing results]
It is classical that on a compact spin manifold, positivity of the scalar curvature implies the vanishing of the $\hat A$--genus. The vanishing result obtained here under the condition $s_2 \ge 0$ differs in nature, as the $2$--curvature is a sectional curvature associated with the four--index tensor $\dR_2$ and encodes strictly more information than scalar or Ricci curvatures. This shows in particular that higher--rank curvature conditions can yield topological consequences not captured by Ricci contraction alone.
\end{remark}

\section{A higher index Gravity of Lovelock}\label{lovelock}

David Lovelock proposed in 1971 \cite{Lovelock} a generalization of Einstein's field equations to dimensions $n>4$ by modifying only the gravitational sector. Matter in Lovelock's theory is coupled exactly as in general relativity. Lovelock's field equations read
\begin{equation}\label{Lov-Field}
\sum_{k}\alpha_kT_{2k}=c(n)T.
\end{equation}
Here, the right-hand side $T$ is the stress-energy $(0,2)$-tensor (as in GR), arising from the matter action. The sum ranges over integers $k\geq 0$. One has $T_0=\g$ (the metric tensor), $T_2$ is the Einstein tensor, and the higher tensors $T_{2k}$, for $4\leq 2k<n$, are the Lovelock tensors. For $2k\geq n$, one has $T_{2k}=0$. The coefficients $\alpha_k$ and $c(n)$ are constants.

\subsection{Lovelock tensors and the $(p,q)$-curvature tensors}

The curvature tensors introduced in the previous sections naturally arise in the context of higher-dimensional gravity.
In particular, Lovelock’s theory of gravity is obtained by contracting higher–order curvature tensors constructed from the Riemann curvature tensor. In this section, we introduce the Lovelock tensors and a broader family of curvature tensors within the framework of double forms, and we show how these objects naturally lead to a higher–index formulation of Lovelock’s gravitational field equations.\\

For $2\leq 2q\leq n$, let $R^q$ denote the $q$-fold exterior product of the Riemann curvature tensor $R$. This tensor is sometimes called the $2q$-Thorpe tensor, or the $2q$-Gauss--Kronecker curvature tensor \cite{Thorpe, Kulkarni}. It is a $(0,4q)$ covariant tensor whose components are given by
\[
\R^q_{i_1,\dots,i_{2q},j_1,\dots,j_{2q}}
=\frac{1}{4^k}\delta^{a_1\cdots a_{2q}}_{i_1\cdots i_{2q}}
\delta^{b_1\cdots b_{2q}}_{j_1\cdots j_{2q}}
\prod_{\ell=1}^q\R_{a_{2\ell -1}a_{2\ell}b_{2\ell -1}b_{2\ell}}.
\]

\subsubsection{Lovelock tensors}

The Lovelock tensors $T_{2q}$, for $0\leq 2q\leq n$, are obtained from the $4q$-index Gauss--Kronecker curvature tensor $\R^q$ in the same way as the Einstein tensor is obtained from the usual $4$-index Riemann curvature tensor $\R$. More precisely, they are given by \cite{Labbi-variational,Labbi-Thorpe}
\begin{equation}
T_{2q}=\frac{\cc^{2q}\R^q}{(2q)!}\g-\frac{\cc^{2q-1}\R^q}{(2q-1)!}.
\end{equation}
These are symmetric and divergence free $(1,1)$ double forms. Their components are given by
\[
(T_{2q})_{ij}=\frac{1}{((2q)!)^2}\sum_{\substack{A,\, B \\ |A|=|B|=2q}}\delta^{Ai}_{Bj}\R^q_{(A,B)}.
\]
Here, the sum is over all multi-indices $A,B$ of length $2q$.

For $q=0$, we set $T_0=\g$. For $q=1$, we recover the usual Einstein tensor.

For $q=2$, the tensor
\[
T_{4}=\frac{\cc^{4}R^2}{4!}\g-\frac{\cc^{3}\R^2}{3!}
\]
is obtained by contracting the $8$-index Gauss--Kronecker curvature tensor $R^2$. We shall show in the next proposition that $T_{4}$ can alternatively be obtained by contracting a $4$-index curvature tensor, namely the tensor $\R\circ \dR_2$, where $\dR_2$ is the generalized double dual of Riemann and $\circ$ denotes the composition product of double forms. Precisely, we have the following.

\begin{proposition}\label{lov-index4}
Let $\dR_2$ be the $(2,2)$ double form representing the generalized double dual of Riemann. Then the Lovelock tensor $T_4$ can be obtained by contracting the $(2,2)$ double form $\R\circ \dR_2$ as follows:
\begin{equation}
T_{4}=\frac{1}{2}\cc^2(\R\circ \dR_2)\, \g-2\cc(\R\circ \dR_2).
\end{equation}
\end{proposition}

The above proposition follows directly from the following two lemmas.

\begin{lemma}
Let $h$ be any symmetric $(1,1)$ double form. Then one has
\begin{equation}
\cc(\R\circ \g h)=\iota_h\R+\Ric \circ h.
\end{equation}
Here $\iota_h$ denotes the interior product of double forms (see \cite{Labbi-double-forms, Labbi-Lapl-Lich}).
\end{lemma}

\begin{remark}
We remark that $\iota_h\R=\overset{\circ}{R}$ is the Riemann tensor of the second kind.
\end{remark}

\begin{proof}
Let $k$ be an arbitrary $(1,1)$ double form. Using algebraic properties of the composition product and the inner product of double forms, we compute
\begin{align*}
\langle \cc(\R\circ \g h),k\rangle
=&\ \langle \R\circ \g h,\g k\rangle
= \langle \R,\g k\circ \g h\rangle\\
=&\ \langle \R,\g(k\circ h)+kh\rangle
=\langle \cc \R,k\circ h\rangle+\langle \R,kh\rangle\\
=&\ \langle \Ric \circ h,k\rangle+\langle \iota_h\R,k\rangle.
\end{align*}
In the third equality, we used the Greub--Vanstone identity $\g k\circ \g h=\g(k\circ h)+kh$.
\end{proof}

The next lemma provides an alternative short form of a general Lanczos identity \cite[formula (11)]{Labbi-Thorpe}.

\begin{lemma}
Let $\dR_2$ be the $(2,2)$ double form representing the generalized double dual of Riemann. Then one has
\begin{equation}
\cc^3\R^2=12 \cc(\R\circ \dR_2).
\end{equation}
In particular, $\cc^4\R^2=12 \cc^2(\R\circ \dR_2)$.
\end{lemma}

\begin{proof}
Recall that formula (\ref{dR2}) gives $\dR_2= \R-\g\, \Ric+\frac{1}{4} \g^2\, \Scal$. It follows that
\begin{align*}
\cc(\R\circ \dR_2)
=&\ \cc\left( \R\circ \R-\R\circ \g\Ric+1/2\R.\Scal\right)\\
=&\ \cc( \R\circ \R)-\cc(\R\circ \g\Ric)+1/2\cc(\R.\Scal)\\
=&\ \cc( \R\circ \R)-\iota_{\Ric}\R-\Ric \circ \Ric+1/2\Ric .\Scal\\
=&\ \frac12\left(\frac{\cc^3\R^2}{3!}\right).
\end{align*}
In the last step, we used a generalization of the Lanczos identity \cite[formula (11)]{Labbi-Thorpe}. In the preceding step, we used the previous lemma, which shows that $\cc(\R\circ \g \Ric)=\iota_{\Ric}\R+\Ric \circ \Ric$.
\end{proof}

\subsubsection{The $(p,q)$-curvature tensors}

\begin{definition}[\cite{Labbi-double-forms}]
For integers $1\leq q\leq {n\over 2}$ and $0\leq p\leq n-2q$, the $(p,q)$-curvature tensor, denoted here by $R^{(p,q)}$, is a $(0,2p)$ covariant tensor whose components are given by
\[
\R^{(p,q)}_{i_1,\dots,i_{p},j_1,\dots,j_{p}}
=\frac{1}{((2q)!)^2}\sum_{\substack{A,\, B \\ |A|=|B|=2q}}
\delta^{Ai_1,\dots,i_{p}}_{Bj_1,\dots,j_{p}}\R^q_{(A,B)}.
\]
Here, the sum is over all multi-indices $A,B$ with length $2q$.
\end{definition}

Using a manipulation completely similar to the one in the proof of Proposition~\ref{index-free}, one obtains the following index-free formulation:
\begin{equation}\label{Rpq}
\R^{(p,q)}={1\over (n-2q-p)!}\ast\bigl( g^{n-2q-p}\R^q\bigr),
\end{equation}
or equivalently, using the identity $\ast\, \g^r\, \ast=\cc^r$,
\begin{equation}\label{Cpq}
\R^{(p,q)}={1\over (n-2q-p)!}\,  \cc^{n-2q-p}(\ast\, \R^q\bigr),
\end{equation}
It turns out that the tensors $\R^{(p,q)}$ can also be obtained by contracting directly the tensors $R^q$. Indeed, formula (15) of \cite{Labbi-double-forms} shows that
\begin{equation}\label{dRcR}
\R^{(p,q)}=\sum_{r=\max\{0,2q-p\}}^{2q}\frac{(-1)^r}{r!}\frac{\g^{p-2q+r}}{(p-2q+r)!}\cc^r(R^q).
\end{equation}

For $p=0$ and $2\leq 2q\leq n$, the tensors $\R^{(0,q)}$ are scalar functions and are, up to a constant factor, the Gauss--Bonnet curvatures of $(M,g)$. In particular, when $n$ is even, $\R^{(0,n/2)}$ is the Gauss--Bonnet integrand.

For $q=1$ and $0\leq p\leq n-2$, one has $\R^{(p,1)}=\dR_p$, which coincides with the generalized double duals of Riemann introduced above.

For $p=1$ and $2q\leq n$, $\R^{(1,q)}$ coincides with the $2q$-Lovelock tensor $T_{2q}$.

\medskip

The curvature tensors $\R^{(p,q)}$ are symmetric $(p,p)$ double forms satisfying the first Bianchi identity. In general, they do not satisfy the second Bianchi identity, but they are always divergence free. Moreover, they satisfy the following hereditary property with respect to contractions (for $p\geq 1$):
\begin{equation}
\cc \R^{(p,q)}=(n-2q-p)\, \R^{(p-1,q)}.
\end{equation}
In particular, $(p-1)$ contractions of $\R^{(p,q)}$ generate the Lovelock tensor $T_{2q}$ of order $2q$, and $p$ contractions generate the $(2q)$-Gauss--Bonnet curvature $h_{2q}$:
\[
\R^{(p,q)} \xrightarrow{\cc} \R^{(p-1,q)} \xrightarrow{\cc } \cdots
\xrightarrow{\cc }\R^{(1,q)}={\mathbf{Lovelock\,\, tensor}} \xrightarrow{\cc } {\mathrm{ GB \, curvature}.}
\]
Because of this hereditary property, we shall call $\R^{(p,q)}$ the \emph{$p$-parent} of the $(2q)$-Lovelock tensor $T_{2q}$.
\begin{remark}
The tensors $\R^{(p,q)}$ are used in \cite{Labbi-Thorpe} to define the so called $2q$-Thorpe and $2q$-anti-Thorpe metrics. These are generalizations of the Einstein condition and conformal flateness. Precisely, if the dimension of the manifold $n=2p$ is even, we have the following analogous of Proposition \ref{self}: 
\begin{itemize}
\item The self duality condition $\ast(\R^{(p,q)})=\R^{(p,q)}$ is equivalent to the manifold being $(2q)$-Thorpe.
\item The anti-self duality condition $\ast(\R^{(p,q)})=-\R^{(p,q)}$ is equivalent to the manifold being $(2q)$-anti-Thorpe.
\end{itemize}
\end{remark}
\subsection{Parent Lovelock's field equations}\label{subs}

Recall the orthogonal decomposition of the $(2q,2q)$ double form $R^q$ \cite{Labbi-double-forms}:
\[
R^q=\omega_{2q}+\g \omega_{2q-1}+ \cdots +\g^{2q-2}\omega_2+\g^{2q-1}\omega_1+\g^{2q}\omega_0,
\]
where each $\omega_i$ is a trace free $(i,i)$ double form.

Corollary 4.4 of \cite{Labbi-double-forms} shows that for $n\geq 2q$ and $p\leq n-2q$ one has
\begin{equation}\label{decomp}
(n-2q-p)!\, \R^{(p,q)}
=\ast\bigl( g^{n-2q-p}\R^q\bigr)
=\sum_{i=0}^{\min\{2q,p\}}(-1)^i\frac{(n-p-i)!}{(p-i)!}\g^{p-i}\omega_i.
\end{equation}
In particular, for $p=1$ we obtain
\begin{equation}
(n-2q-1)!\, T_{2q}
=\sum_{i=0}^{1}(-1)^i\frac{(n-1-i)!}{(1-i)!}\g^{1-i}\omega_i.
\end{equation}

The above analysis shows that Lovelock’s field equations \eqref{Lov-Field} retain only a limited part of the curvature information encoded in the tensor $R^q$. More precisely, except for the components $\omega_0$ and $\omega_1$ appearing in the orthogonal decomposition of $R^q$, all higher trace-free components are systematically eliminated by the contraction process defining the Lovelock tensor $T_{2q}$. From this perspective, Lovelock gravity can be viewed as a theory in which most curvature components of $R^q$ do not contribute effectively to the gravitational 

We see that only the components $\omega_0$ and $\omega_1$ of $\R^q$ contribute effectively to $T_{2q}$; the remaining components are lost during the contraction process. In particular, these components do not participate effectively in Lovelock's field equation \eqref{Lov-Field}.

To ensure that all components of $\R^q$ contribute effectively to the field equations, it is natural to replace the Lovelock tensor by one of its $p$-parent tensors. The following proposition shows that this requirement is satisfied precisely for those integers $p$ such that
$p\geq \min\{n-2q,2q\}.$.

\begin{proposition}\label{effective}
All the components $\omega_i$ of $\R^q$ contribute effectively in $\R^{(p,q)}$ if and only if $p\geq \min\{2q,n-2q\}$.
\end{proposition}

\begin{proof}
One has to be careful in dimensions where $2q>n/2$, since in that case some components of $\R^q$ vanish for dimensional reasons and therefore do not contribute even in $R^q$ itself (as in the vanishing of the Weyl component of $\R$ in dimension $3$).

Assume first that $n<4q$. Then $n-2q<2q$ and one has $\R^q=\sum_{i=0}^{n-2q}\g^{2q-i}\omega_i$; see \cite[Proposition 2.1]{Labbi-Gen-Einstein}. The remaining components $\omega_i$ for $i>n-2q$ vanish for dimensional reasons. In this case, \eqref{decomp} is equivalent to
\begin{equation}\label{decomp2}
(n-2q-p)!\, \R^{(p,q)}
=\ast\bigl( g^{n-2q-p}\R^q\bigr)
=\sum_{i=0}^{\min\{n-2q,p\}}(-1)^i\frac{(n-p-i)!}{(p-i)!}\g^{p-i}\omega_i.
\end{equation}
It follows that the statement holds in this case if and only if $p\geq n-2q$.

Next, assume that $n\geq 4q$. Then the statement follows directly from \eqref{decomp} provided $p\geq 2q$.
\end{proof}

Let us emphasize the two principal cases.

\begin{proposition}
\begin{itemize}
\item For $n\geq 4q$, one has
\[
\R^q=\sum_{i=0}^{2q}\g^{2q-i}\omega_i,\,\,\, {\mathrm{and}}\,\,\, 
\R^{(2q,q)}=\sum_{i=0}^{2q}(-1)^i\binom{n-2q-i}{2q-i}\g^{2q-i}\omega_i.
\]
In particular, all the components $\omega_i$ in $\R^q$ survive in $\R^{(2q,q)}$. Furthermore, the two tensors share the same trace free part $\omega_{2q}$.

\item For $n<4q$, one has
\[
\R^q=\sum_{i=0}^{n-2q}\g^{2q-i}\omega_i,\,\,\, {\mathrm{and}}\,\,\, 
\R^{(n-2q,q)}=\sum_{i=0}^{n-2q}(-1)^i\frac{(2q-i)!}{(n-2q-i)!}\g^{n-2q-i}\omega_i.
\]
In particular, all the nonvanishing components $\omega_i$ in $\R^q$ survive in $\R^{(n-2q,q)}$.
\end{itemize}
\end{proposition}

In conclusion, for all components of $R^q$ to contribute effectively in the field equations, one may replace $T_{2q}$ in Lovelock's equation by its parent tensor $\R_{(2q,q)}$ in the case $2q\leq n/2$, and by the parent tensor $\R_{(n-2q,q)}$ in the case $2q>n/2$.

One possible way to write a full gravitational field equation at a higher index level, in which all components $\omega_i$ of $R^q$ contribute effectively, is
\[
\sum_{0\leq 2q\leq n/2}\alpha_{2q}\R_{(2q,q)}
+\sum_{n\geq 2q>n/2}\alpha_{2q}\R_{(n-2q,q)}
=\chi \mathbb{T}.
\]
However, the left-hand side of the above equation is not homogeneous, since it sums tensors of different valences. An alternative is to take a common higher $d$-parent so that all terms become fixed $(d,d)$ double forms:
\[
\sum_{q}\alpha_{q}\R_{(d,q)}=\chi \mathbb{T}.
\]
Here $d=d(n)$ is a constant depending only on the dimension $n$. It is chosen according to Proposition~\ref{effective} as the maximum of the set
\[
\{\min\{2q,n-2q\}:2\leq 2q<n\}.
\]
It is not difficult to check that $d(n)$ is given by
\[
d(n)=
\begin{cases}
\frac{n-1}{2} & n \,\, {\mathrm{odd}},\\[2mm]
n/2 & n\equiv 0 \pmod{4},\\[2mm]
\frac{n-2}{2} & n\equiv 2 \pmod{4}.
\end{cases}
\]

\end{document}